\documentclass[a4paper]{article}
\usepackage[dvips]{graphicx}
\usepackage{amsmath, amssymb, amsfonts, amsthm}
\usepackage{ascmac}
\newtheorem{defn}{Definition}
\newtheorem{thm}{Theorem}

\newtheorem{lem}{Lemma}
\newtheorem*{rem}{Remark}

\newtheorem{prop}{Proposition}[section]

\begin{document}

\title{The growth rates of ideal Coxeter polyhedra \\ in hyperbolic 3-space}

\author{Jun Nonaka}





\date{}

\maketitle
\begin{abstract}
In \cite{ref:kp}, Kellerhals and Perren conjectured that the growth rates of the 
reflection groups given by hyperbolic Coxeter polyhedra are always Perron numbers. 
We prove that this conjecture is always true for the case of ideal Coxeter polyhedra in $\mathbb{H}^3$. 
We also find out the ideal Coxeter polyhedron in $\mathbb{H}^3$ with the smallest growth rate. 
Finally, we show that there are correlations between 
the volumes and the growth rates of ideal Coxeter polyhedra in $\mathbb{H}^3$ in many cases. 
\end{abstract}



\section{Introduction} 
A convex polyhedron is called a {\it Coxeter polyhedron} if its dihedral angles are submultiples of $\pi$. 
In particular, if all of its dihedral angles are $\pi /2$, then we call this polyhedron {\it right-angled}. \\
Any Coxeter polyhedron is a fundamental domain of the discrete group $\Gamma$ generated by the set $S$ 
consisting of the reflections with respect to its facets. 
We call $(\Gamma ,S)$ an $n$-$dimensional$ $hyperbolic$ $Coxeter$ $group$ 
if $(\Gamma ,S)$ is obtained by a Coxeter polyhedron of finite volume in hyperbolic $n$-space $\mathbb{H}^n$. 
In this situation, we define the $word$ $length$ $l_{S}(x)$ $of$ $x\in \Gamma$ $with$ $respect$ $to$ $S$ by 
the smallest integer $i\geq 0$ for which there exist $s_1$, $s_2$, $\cdots $, $s_i\in S$ 
such that $x=s_1s_2\cdots s_i$. 
The growth function $f_S(t)$ of $(\Gamma ,S)$ is the formal power series $\sum_{j=0}^{\infty }a_jt^j$ where 
$a_j$ is the number of elements $g\in \Gamma $ satisfying $l_S(g)=j$. 
We also define the $growth$ $rate$ of $(\Gamma ,S)$ by $\tau :=\limsup_{j\to \infty }\sqrt[j]{a_j}$. 

There are some results about growth rates of hyperbolic Coxeter groups. 
For example, Cannon-Wagreich and Parry showed that 
the growth rates of two and three-dimensional cocompact hyperbolic Coxeter groups are Salem numbers 
\cite{ref:cw, ref:pa}, where a real algebraic integer $\tau >1$ is called a {\it Salem number} if $\tau ^{-1}$ is 
an algebraic conjugate of $\tau $ and the absolute value of all algebraic conjugates of $\tau $ other than 
$\tau ^{-1}$ is 1. 
Kellerhals and Perren proved that the growth rates of four-dimensional cocompact hyperbolic Coxeter groups 
with at most 6 generators are Perron numbers \cite{ref:kp}. 
A real algebraic integer $\tau '>1$ is called {\it Perron number} if its conjugates have moduli less than the 
modulus of $\tau '$. 

In non-compact case, Floyd proved that the growth rates of two-dimensional non-cocompact hyperbolic 
Coxeter groups are {\it Pisot-Vijayaraghavan numbers}, where a real algebraic integer $\tau ''$ is called 
a Pisot-Vijayaraghavan number if the absolute value of all algebraic conjugates of $\tau ''$ other than $\tau ''$ 
is less than or equal to 1. 
Note that both a Salem number and a Pisot-Vijayaraghan number are also Perron numbers. 
Komori and Umemoto proved that the growth rates of three-dimensional hyperbolic Coxeter groups 
with four or five generators are Perron numbers \cite{ref:ku, ref:u}. 

Regarding these results, Kellerhals and Perren conjectured that 
the growth rates of hyperbolic Coxeter groups are always Perron numbers. 
In this paper, we prove that the growth rates of three-dimensional hyperbolic Coxeter groups 
given by ideal Coxeter polyhedra with finite volume are Perron numbers. 
The definition of an ideal Coxeter polyhedron is as follows. 
\begin{defn}
{\rm A hyperbolic polyhedron is} ideal 
{\rm if all of its vertices are ideal points of $\mathbb{H}^n$, i.e. all of them belongs to 
$\partial \overline{\mathbb{H}^n}$. }
\end{defn}
Note that an ideal Coxeter polyhedron is non-compact. 

In compact case, 
Kellerhals and Kolpakov found the minimal growth rate of Coxeter polyhedra in $\mathbb{H}^3$ \cite{ref:kk}. 
In the same paper, they also consider the relation between its volumes and its growth rates of 
Coxeter polyhedra. 
In Section 5, we detect the ideal Coxeter polyhedron in $\mathbb{H}^3$ with the minimal growth rate. 
Moreover, in Section 6, we see the relation between the volumes and the growth rates of 
some ideal Coxeter polyhedra. 

\section{Ideal Coxeter polyhedra} 
In this section, we focus on some characteristics of ideal Coxeter polyhedra of finite volume 
in $\mathbb{H}^3$. 

Felikson and Tumarkin proved that 
a simple ideal hyperbolic Coxeter polyhedron can exist only in $\mathbb{H}^n$ for $n< 8$ in \cite{ref:ft}. 
On the other hand, Kolpakov proved that 
there is no ideal right-angled polyhedron in $\mathbb{H}^n$ for $n\geq 7$ in \cite{ref:k}. 
In the same paper, he also proved that 24-cell is the ideal right-angled polyhedron 
which has the minimal facet number in $\mathbb{H}^4$. \\
In $\mathbb{H}^3$, it is well-known that an octahedron is the ideal right-angled polyhedron which has 
the minimal facet number among all ideal Coxeter polyhedra. 
We give a short proof in the end of this section. \\
From now on, we restrict our attention to ideal Coxeter polyhedra of finite volume in $\mathbb{H}^3$. 
When an ideal Coxeter polyhedron has finite volume, we call its vertex $cusp$. 
The following lemma plays very important role in the proof of our main theorems. 
\begin{lem}\label{lem-1}
Any cusp of a three-dimensional hyperbolic ideal Coxeter polyhedron satisfies one of the following conditions: \\
$(i)$ shared by four faces, and the dihedral angles of these faces are equal to $\frac{\pi}{2}$, \\
$(ii)$ shared by three faces, and the dihedral angles of these faces are equal to $\frac{\pi}{3}$, \\
$(iii)$ shared by three faces, and one of the dihedral angles of these faces is $\frac{\pi}{2}$ and the other two 
dihedral angles are $\frac{\pi}{4}$, \\
$(iv)$ shared by three faces, and the dihedral angles of these faces are $\frac{\pi}{2}$, $\frac{\pi}{3}$ 
and $\frac{\pi}{6}$. 
\end{lem}
\begin{proof}
Fix one cusp of an ideal Coxeter polyhedron. 
Consider the upper half-space model whose point at infinity is the cusp we fixed. 
We may assume that this cusp is shared by $k$ faces. 
The hyperplanes containing these faces are vertical Euclidean 
planes which intersects the boundary of the upper half-space orthogonally. 
In this model, the dihedral angles of these faces are equal to the dihedral angles of corresponding 
hyperplanes in the Euclidean sense. 
We denote their dihedral angles by $\frac{\pi}{m_1}$, $\frac{\pi}{m_2}$, $\cdots$, $\frac{\pi}{m_k}$. 
Then we obtain 
\begin{equation*}
\sum_{i=1}^k\left( \pi -\frac{\pi}{m_i}\right)=2\pi. 
\end{equation*}
Thus 
\begin{equation*}
\frac{1}{m_1}+\cdots +\frac{1}{m_k}=k-2.
\end{equation*}
We may assume that $2\leq m_1\leq m_2\leq \cdots \leq m_k$. Then, by using the above equality, 
we obtain  
\begin{equation*}
\frac{k}{m_1}\geq k-2.
\end{equation*}
Thus 
\begin{equation*}
k\leq \frac{2m_1}{m_1-1}=2+\frac{2}{m_1-1} \leq 2+\frac{2}{2-1}=4.
\end{equation*}
This means that $k$ is equal to $3$ or $4$. 
When $k=4$, $m_1=m_2=m_3=m_4=2$. This corresponds to Condition $(i)$. 
When $k=3$, $(m_1, m_2, m_3)=(3,3,3)$, $(2,4,4)$ or $(2,3,6)$. These correspond to Conditions 
$(ii)$, $(iii)$ and $(iv)$ respectively. 
\end{proof}
From now on, we express $the$ $cusp$ $of$ $type$ $(i)$ $($resp. $(ii)$, $(iii)$ and $(iv))$ 
if the cusp satisfies condition $(i)$ $($resp. $(ii)$, $(iii)$ and $(iv))$ in Lemma \ref{lem-1}. 
We denote by $c_8$ $($resp. $c_9$, $c_{10}$ and $c_{11})$ 
the number of cusps of type $(i)$ $($resp. $(ii)$, $(iii)$ and $(iv))$. 
We also denote the total number of faces, edges and cusps by $f$, $e$ and $c$ respectively. 
Any edge of an ideal Coxeter polyhedron in $\mathbb{H}^3$ is shared by exactly two faces. 
If the dihedral angle of these two faces is $\frac{\pi}{m}$, we call this edge $\frac{\pi}{m}$-$edge$. 
We denote by $e_m$ the number of $\frac{\pi}{m}$-edges of an ideal Coxeter polyhedron. 
Then we obtain the following combinatorial identities. 
\begin{equation}
c-e+f=2,
\end{equation} 
\begin{equation}
4c_8+3c_9+3c_{10}+3c_{11}=2e,
\end{equation}
\begin{equation}
2e_2=4c_8+c_{10}+c_{11},
\end{equation}
\begin{equation}
2e_3=3c_9+c_{11},
\end{equation}
\begin{equation}
e_4=c_{10},
\end{equation}
\begin{equation}
2e_6=c_{11}.
\end{equation}
The first identity is Euler's identity. 
The other identities are also obtained by seeing both the number of edges sharing each cusp 
and that our assumption under consideration is ideal. 
For example, the identity $(2)$ is obtained as follows. 
Any cusp of type $(i)$ is shared by four edges and the other cusps are shared by three edges. 
On the other hand, any edge has two cusps. Thus we obtain the identity $(2)$. 

Moreover, by Lemma \ref{lem-1}, we obtain the following trivial identities; 
\begin{equation}
c=c_8+c_9+c_{10}+c_{11},
\end{equation}
\begin{equation}
e=e_2+e_3+e_4+e_6. 
\end{equation}
By using identities $(1)$, $(2)$ and $(7)$, we obtain 
\begin{equation}
c_8=2f-c-4.
\end{equation}
By substituting $(9)$ in $(7)$, we get 
\begin{equation}
c_9+c_{10}+c_{11}=2c-2f+4. 
\end{equation}

If an ideal Coxeter polyhedron in $\mathbb{H}^3$ is right-angled, then the type of any cusp is $(i)$, 
i.e. $c=c_8$. 
Thus, by $(9)$, we obtain $c=f-2$ in this case. 
On the other hand, since any face has at least three cusps 
while the each cusp is shared by at most four faces, we also obtain $4c\geq 3f$. 
Hence, if the ideal Coxeter polyhedron in $\mathbb{H}^3$ is right-angled, then we obtain $f\geq 8$. 
This means that the polyhedron with the minimal facet number among all ideal right-angled polyhedra 
in $\mathbb{H}^3$ must be an octahedron. 

\section{Growth functions} 
In this section, we consider the growth functions associated to 
hyperbolic Coxeter polyhedra in $\mathbb{H}^3$. 
In order to calculate the growth rates, we use the following theorems. 
\begin{thm}[\cite{ref:sol}]\label{thm-1}
The growth function ${\cal F} _S(t)$ of an irreducible spherical Coxeter group $(\Gamma , S)$ can be written as 
${\cal F}_s(t)=\Pi _{i=1}^k[m_i+1]$ where $[n]:=1+t+\cdots +t^{n-1}$ and $\{ m_1, m_2, \cdots , m_k\}$ is the set of exponents $($as defined in Section 3.16 of \cite{ref:hu}$)$ of $(\Gamma , S)$. 
\end{thm}
\begin{thm}[\cite{ref:ste}]\label{thm-2}
Let $(\Gamma , S)$ be a Coxeter group. Let us denote the Coxeter subgroup of $(\Gamma , S)$ 
generated by the subset $T\subseteq S$ by $(\Gamma _T, T)$, and denote its growth function by 
${\cal F}_T(t)$. Set $F=\{ T\subseteq S \mid \Gamma _T\ is\ finite\}$. Then 
\begin{equation*}
\frac{1}{{\cal F}_S(t^{-1})}=\sum_{T\in F}\frac{(-1)^{|T|}}{{\cal F}_T(t)}.
\end{equation*}
\end{thm}
The following list of the exponents of the growth series are obtained in \cite{ref:hu,ref:kp}, 
where $[n]:=1+t+\cdots +t^{n-1}$ and $[n,m]=[n][m]$. 
\begin{table}[htb]
\begin{center}
  \begin{tabular}{|c|c|c|} \hline
  Group symbols & Exponents & $f_s(t)$  \\ \hline \hline
    $A_n$ & $1,2,\cdots ,n$ & $[2,3,\cdots ,n+1]$  \\ \hline
    $B_n$ & $1,3,\cdots ,2n-1$ & $[2,4,\cdots ,2n]$ \\ \hline
    $D_n$ & $1,3,\cdots ,2n-3,n-1$ & $[2,4,\cdots ,2n-2][n]$ \\ \hline
    $E_6$ & $1,4,5,7,8,11$ & $[2,5,6,8,9,12]$ \\ \hline
    $E_7$ & $1,5,7,9,11,13,17$ & $[2,6,8,10,12,14,18]$ \\ \hline
    $E_8$ & $1,7,11,13,17,19,23,29$ & $[2,8,12,14,18,20,24,30]$ \\ \hline
    $F_4$ & $1,5,7,11$ & $[2,6,8,12]$ \\ \hline
    $H_3$ & $1,5,9$ & $[2,6,10]$ \\ \hline
    $H_4$ & $1,11,19,29$ & $[2,12,20,30]$ \\ \hline
    $I_2(m)$ & $1,m-1$ & $[2,m]$ \\ \hline
  \end{tabular}
  \caption{Exponents and growth polynomials of irreducible finite Coxeter groups.}
\end{center}
\end{table} 
Denote by ${\cal F}_{\langle P\rangle}(t)$ the growth function of the reflection group given by 
an ideal Coxeter polyhedron $P$ in $\mathbb{H}^3$. 
Since $P$ has no vertices in $\mathbb{H}^3$, 
any non-trivial finite Coxeter subgroup of a Coxeter group obtained by $P$ 
is generated by either exactly one generator represented by a face of $P$ 
or generators represented by faces whose intersection is an edge of $P$. 

The symbol of the Coxeter subgroup obtained by exactly one generator is $A_1$. 
Thus, the growth function of this subgroup is $[2]=1+t$. 
Note that the number of Coxeter subgroups whose symbols are $A_1$ is $f$. 

The symbol of the Coxeter subgroup obtained by an $\frac{\pi}{2}$-edge of $P$ is $A_1\times A_1$. 
Thus, the growth function of this subgroup is $[1,1]=(1+t)^2$. 

In the same manner, 
we obtain the growth function of each Coxeter subgroup of the Coxeter group obtained by $P$. 
Thus, by Theorem \ref{thm-2}, we obtain
\begin{eqnarray*}
\frac{1}{{\cal F}_{\langle P\rangle}(t^{-1})}&=&1-\frac{f}{1+t}+\frac{e_2}{(1+t)^2}+\frac{e_3}{(1+t)(1+t+t^2)}+\frac{e_4}{(1+t)(1+t+t^2+t^3)} \\
&\ &+\frac{e_6}{(1+t)(1+t+t^2+t^3+t^4+t^5)} \\
&=&1-\frac{f}{1+t}+\frac{6f-2c-c_9-12}{2(1+t)^2}+\frac{2c-2f+2c_9-c_{10}+4}{2(1+t)(1+t+t^2)}\\
&\ &+\frac{c_{10}}{(1+t)^2(1+t^2)}+\frac{2c-2f-c_9-c_{10}+4}{2(1+t)^2(1+t+t^2)(1-t+t^2)}. 
\end{eqnarray*}
The second equality is obtained by using the identities $(3$-$6)$, $(9)$ and $(10)$. 
By simplfying the expression above, we get the following growth function's formula: 
\begin{equation*} 
{\cal F}_{\langle P\rangle}(t)=\frac{2(1+t)^2(1+t^2)(1+t+t^2)(1-t+t^2)}{(t-1)g_{\langle P\rangle}(t)}.
\end{equation*}
The function $g_{\langle P\rangle}(t)$ is as follows: 
\begin{eqnarray*}
g_{\langle P\rangle}(t)&:=&(2c-2)t^7+(2c-2f+2)t^6+(2c+2f-c_9-8)t^5\\
&\ &+(4c-4f+c_9-c_{10}+4)t^4+(4f-c_9+c_{10}-12)t^3\\
&\ &+(2c-2f+c_9)t^2+(2f-6)t-2. 
\end{eqnarray*}
By Cauchy-Hadamard's theorem, 
the radius of convergence of ${\cal F}_{\langle P\rangle}(t)$ is the reciprocal of the growth rate of $P$. 
Thus, to see the growth rate of $P$, it suffices to find out the smallest absolute value of the solution of 
$g_{\langle P\rangle}(t)=0$. \\
By substituting $t=0$ in $g_{\langle P\rangle}(t)$, we obtain $g_{\langle P\rangle}(0)=-2$. 
On the other hand, when $t=\frac{1}{2}$, 
\begin{equation*}
g_{\langle P\rangle}\left( \frac{1}{2}\right) =\frac{55c+50f+10c_9+4c_{10}-415}{64}. 
\end{equation*}
If $P$ is not a simplex, then the right-hand side of the above equation is positive since both $c$ and $f$ 
are at least 5. 

Now we consider the case that $P$ is a simplex. 
In this case, by Lemma \ref{lem-1}, the type of any cusp of $P$ is one of $(ii)$, $(iii)$ and $(iv)$. 

If $P$ has a cusp of type $(iv)$, then there is a $\frac{\pi }{6}$-edge. 
Both endpoints of this edge are cusps of type $(iv)$. 
Additionally, the other two cusps of $P$ also have to be cusps of type $(iv)$ by Lemma \ref{lem-1}. 
Thus, if the ideal simplex $P$ has a cusp of type $(ii)$, 
then any endpoint of $\frac{\pi }{3}$-edges is a cusp of type $(ii)$ since 
$P$ does not have a cusp of type $(iv)$. 
Therefore, any cusp of the ideal simplex $P$ is type $(iii)$ in the case that $P$ has a cusp of type $(iii)$. 

Hence if $P$ is a simplex, then there are three possibilities.  
The first one is that $P$ has two $\frac{\pi}{2}$-edges, two $\frac{\pi}{3}$-edges and two $\frac{\pi}{6}$-egdes. 
The second one is that $P$ has two $\frac{\pi}{2}$-edges and four $\frac{\pi}{4}$-edges. 
The last one is that $P$ has six $\frac{\pi}{3}$-edges. 
We can see them in the list of three-dimensional hyperbolic simplices written in \cite{ref:vin}. 
Their combinatorial structures are depicted in Fig.1; 
each circle represents a cusp and 
an edge labelled with $i$ $(i=2$, $3$, $4$, $6)$ is the $\frac{\pi}{i}$-edge. 
If $P$ is a simplex, then $c=f=4$. Moreover, 
if the simplex $P$ has two $\frac{\pi}{2}$-edges, two $\frac{\pi}{3}$-edges 
and two $\frac{\pi}{6}$-edges, then $c_{9}=c_{10}=0$. 
If $P$ is a simplex having two $\frac{\pi}{2}$ edges and four $\frac{\pi}{4}$ edges, 
then $c_{10}=4$ and $c_9=0$. 
When $P$ is a simplex which has four $\frac{\pi}{3}$ edges, then $c_9=4$ and $c_{10}=0$. 
Eventually, if $P$ is a simplex, then $g_{\langle P\rangle }\left( \frac{1}{2}\right) $ is also positive. 

By $g_{\langle P\rangle}(0)=-2$, $g_{\langle P\rangle}(\frac{1}{2})>0$ and 
the continuity of $g_{\langle P\rangle}(t)$ for $t$, 
the equation $g_{\langle P\rangle}(t)=0$ has at least one solution in the interval $\left( 0, \frac{1}{2}\right)$. 

In the next section, 
we prove that this solution is the unique solution of $g_{\langle P\rangle}(t)=0$ for $0<t<\frac{1}{2}$. 
In addition, we prove that the absolute value of this solution is smaller than 
any other solutions of $g_{\langle P\rangle}(t)=0$. 

\begin{figure}[htp]
\begin{center}
\begin{picture}(220, 60)
\put(0,0){\circle{5}}
\put(60,0){\circle{5}}
\put(30,60){\circle{5}}
\put(30,22){\circle{5}}

\put(3,0){\line(1,0){54}}
\put(2,2){\line(4,3){25.5}}
\put(58,2){\line(-4,3){25.5}}
\put(30,24){\line(0,1){34}}
\put(0,3){\line(1,2){28}}
\put(60,3){\line(-1,2){28}}

\put(6,27){$2$}
\put(23,30){$6$}
\put(48,27){$3$}
\put(16.5,5){$3$}
\put(36.5,5){$2$}
\put(27,-9){$6$}

\put(80,0){\circle{5}}
\put(140,0){\circle{5}}
\put(110,60){\circle{5}}
\put(110,22){\circle{5}}

\put(83,0){\line(1,0){54}}
\put(82,2){\line(4,3){25.5}}
\put(138,2){\line(-4,3){25.5}}
\put(110,24){\line(0,1){34}}
\put(80,3){\line(1,2){28}}
\put(140,3){\line(-1,2){28}}

\put(86,27){$4$}
\put(103,30){$2$}
\put(128,27){$4$}
\put(96.5,5){$4$}
\put(116.5,5){$4$}
\put(107,-9){$2$}

\put(160,0){\circle{5}}
\put(220,0){\circle{5}}
\put(190,60){\circle{5}}
\put(190,22){\circle{5}}

\put(163,0){\line(1,0){54}}
\put(162,2){\line(4,3){25.5}}
\put(218,2){\line(-4,3){25.5}}
\put(190,24){\line(0,1){34}}
\put(160,3){\line(1,2){28}}
\put(220,3){\line(-1,2){28}}

\put(166,27){$3$}
\put(183,30){$3$}
\put(208,27){$3$}
\put(176.5,5){$3$}
\put(196.5,5){$3$}
\put(187,-9){$3$}

\end{picture}
\end{center}
\caption{Combinatorial structures of ideal simplices in $\mathbb{H}^3$}
\end{figure}
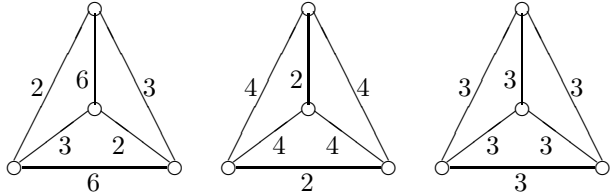


\section{Growth rates of ideal Coxeter polyhedra}
The aim of this section is to prove the following theorem. 
\begin{thm}\label{thm-3}
The growth rate of the reflection group given by an ideal Coxeter polyhedron 
in $\mathbb{H}^3$ is a Perron number. 
\end{thm}
By observation in the previous section, to show Theorem \ref{thm-3}, 
we have to prove that the solution of $g_{\langle P\rangle}(t)=0$ for $t\in \left( 0,\frac{1}{2}\right)$ 
is the unique solution which has the smallest absolute value, where $P$ is an ideal Coxeter polyhedron in 
$\mathbb{H}^3$. 
First of all, we consider a three-dimensional ideal Coxeter polyhedron $Q$ which is not right-angled. 
Since $g_{\langle Q\rangle}(0)<0$ and $g_{\langle Q\rangle}\left( \frac{1}{2}\right) >0$, 
the following Proposition shows that the equation $g_{\langle Q\rangle}(t)=0$ has the unique solution 
in $\left( 0,\frac{1}{2}\right)$ which has the smallest absolute value of the solutions. 
\begin{prop}\label{prop-1}
Let $Q$ be an ideal Coxeter polyhedron in $\mathbb{H}^3$ which is not right-angled. 
Then $g_{\langle P\rangle }(t)$ has the following three properties. 

$(1)$ $g_{\langle Q\rangle}(t)<0$ for $-\frac{1}{2}<t<0$, 

$(2)$ $\frac{\rm d}{\rm dt}g_{\langle Q\rangle}(t)>0$ for $0<t<\frac{1}{2}$, 

$(3)$ the absolute value of any complex $($not real$)$ solution of $g_{\langle Q\rangle}(t)=0$ is greater than 
a real number $r_0\in \left( 0,\frac{1}{2}\right) $ which satisfies $g_{\langle Q\rangle}(r_0)=0$. 
\end{prop}
\begin{rem}
{\rm If $Q$ is a right-angled polyhedron, then there may exist $t_0\in \left( 0, \frac{1}{2}\right)$ such that 
$\left. \frac{\rm d}{\rm dt}g_{\langle Q\rangle}(t)\right|_{t=t_0} \leq 0$. 
That is why we exclude the case that $Q$ is right-angled in Proposition 1. }
\end{rem}
Before proving Proposition \ref{prop-1}, we need to prove the following lemma. 
\begin{lem}\label{lem-2}
Let $f$ be the number of faces of a three-dimensional hyperbolic ideal Coxeter polyhedron 
which is not right-angled, and $c$ the number of its cusps. Then the following identities hold. 
\begin{equation}
2f-c-4\geq 0,
\end{equation}
\begin{equation}
c\geq f.
\end{equation}
\end{lem}
To prove this lemma, we use the following lemma.  
\begin{lem}[\cite{ref:and}]\label{thm:And} 
An acute-angled almost simple polyhedron of finite volume with given dihedral angles, 
other than a tetrahedron or a triangular prism, exists in $\mathbb{H}^3$ 
if and only if the following conditions are satisfied$:$ 

$(a)$ if three faces meet at a vertex, then the sum of the dihedral angles 
between them is greater than $\pi ;$ 

$(b)$ if three faces meet at a cusp, then the sum of the dihedral angles 
between them is equal to $\pi ;$ 

$(c)$ if four faces meet at a vertex or a cusp, then all the 
dihedral angles between them equal $\frac{\pi}{2} ;$ 

$(d)$ if three faces are pairwise adjacent but share neither a vertex nor a cusp, 
then the sum of the dihedral angles between them is less than $\pi ;$ 

$(e)$ if a face $F_i$ is adjacent to faces $F_j$ and $F_k$, 
while $F_j$ and $F_k$ are not adjacent but have a common cusp which $F_i$ 
does not share, then at least one of the angles formed by 
$F_i$ with $F_j$ and with $F_k$ is different from $\frac{\pi}{2} ;$ 

$(f)$ if four faces are cyclically adjacent but meet at neither a vertex nor a cusp, 
then at least one of the dihedral angles between them is different from $\frac{\pi}{2}$. 
\end{lem}
A hyperbolic polyhedron in $\mathbb{H}^3$ is $almost$ $simple$ if any of its edges belongs only to 
two faces and any of its vertices not at infinity belongs only to three faces. 
Thus, any ideal Coxeter polyhedron in $\mathbb{H}^3$ is almost simple. 
\begin{proof}[Proof of Lemma \ref{lem-2}]
{\rm By identity $(9)$, it is clear that inequality (11) holds. 
To prove the inequality (12), we consider three possibilities: 

$c_9\geq 1$, 

$c_{10}\geq 1$, 

$c_{11}\geq 1$ and $c_9=0$. 

If an ideal Coxeter polyhedron satisfies the case that $c_9\geq 1$, 
then there exists a cusp of type $(ii)$. 
Denote this cusp by $v_1$. The cusp $v_1$ is shared by three $\frac{\pi}{3}$-edges. 
By Lemma 1, the endpoints of these three edges other than $v_1$ are either type $(ii)$ or $(iii)$. 
Thus $c_9+c_{11}\geq 4$. 

If an ideal Coxeter polyhedron satisfies the case that $c_{10}\geq 1$, 
then $c_{10}\geq 3$ since both endpoints of 
$\frac{\pi}{4}$-edge are counted in $c_{10}$. If $c_{10}$ is exactly 3, 
then the polyhedron has a triangle whose edges are $\frac{\pi}{4}$-edge 
since any endpoint of $\frac {\pi }{4}$-edges is the cusp of type $(iii)$. 

In this case, there are three faces which are adjacent to this triangle. 
These three faces are adjacent to each other and their dihedral angles are $\frac{\pi}{2}$ since any cusp 
shared by a $\frac{\pi}{4}$-edge 
is shared by exactly two $\frac{\pi}{4}$-edges and one $\frac{\pi}{2}$-edge by Lemma \ref{lem-1}. 
This contradicts Lemma \ref{thm:And} $(d)$. Thus $c_{10}\geq 4$. 

Now we consider the case that $c_{11}\geq 1$ and $c_9=0$. 
In this case, there is a $\frac{\pi}{6}$-edge. 
Each of two endpoints of this edge, which is type $(iv)$, are shared by two $\frac{\pi}{3}$-edges. 
Since we are assuming $c_9=0$, there are no cusps of type $(iii)$. 
Thus, the endpoints of those two $\frac{\pi}{3}$-edges are different and they give two cusps of type $(iii)$. 
Hence $c_{11}\geq 4$. 

Thus, in any case, we obtain 
\begin{equation*}
c_9+c_{10}+c_{11}\geq 4.
\end{equation*}
By using this inequality and the identity (10), we obtain 
\begin{equation*}
2c-2f+4=c_9+c_{10}+c_{11}\geq 4. 
\end{equation*}
Hence we obtain the inequality (12).}
\end{proof}
Now we are ready to prove Proposition \ref{prop-1}. \\
\\
\textbf{Proof of Proposition \ref{prop-1} $(1)$}\\ 
Since $-t^3(t-1)c_{10}\leq 0$ for $t<0$, we obtain 
\begin{eqnarray*}
g_{\langle Q\rangle }(t)&=&(2c-2)t^7+(2c-2f+2)t^6+(2c+2f-c_9-8)t^5\\
&\ &+(4c-4f+c_9-c_{10}+4)t^4+(4f-c_9+c_{10}-12)t^3\\
&\ &+(2c-2f+c_9)t^2+(2f-6)t-2. \\
&=&(1+t^2)(2(c-1)t^5+2(c-f+1)t^4+2(f-3)t^3+2(c-f+1)t^2 \\
&\ &+2(f-3)t-2-t^2(t-1)c_9)-t^3(t-1)c_{10} \\
&\leq &(1+t^2)(2(c-1)t^5+2(c-f+1)t^4+2(f-3)t^3+2(c-f+1)t^2 \\
&\ &+2(f-3)t-2-t^2(t-1)c_9).
\end{eqnarray*}
Define $h_{\langle Q\rangle}(t)$ to be a part of the last expression in the inequality above: 
\begin{equation*}
h_{\langle Q\rangle}(t):= 2(c-1)t^5+2(c-f+1)t^4+2(f-3)t^3+2(c-f+1)t^2+2(f-3)t-2-t^2(t-1)c_9.
\end{equation*}
To prove Proposition \ref{prop-1} $(1)$, 
it is suffices to show that $h_{\langle Q\rangle}(t)<0$ for $-\frac{1}{2}<t<0$. 
Since $0<-t^2(t-1)<\frac{1}{8}$ for $-\frac{1}{2}<t<0$ and $c_9\leq c$, we get an inequality: 
\begin{equation*}
-t^2(t-1)c_9<c. 
\end{equation*}
By this inequality and Lemma \ref{lem-2}, we obtain the upper bound of $h_{\langle Q\rangle }(t)$ as follows: 
\begin{eqnarray*}
h_{\langle Q\rangle}(t)&<&2(c-1)t^5+2(c-f+1)t^4+2(f-3)t^3+2(c-f+1)t^2+2(f-3)t-2+c \\
&\leq &2(f-1)t^5+2(2f-4-f+1)t^4+2(f-3)t^3+2(2f-4-f+1)t^2 \\
&\ &+2(f-3)t-2+2f-5 \\
&=&2(f-1)t^5+2(f-3)t^4+2(f-3)t^3+2(f-3)t^2+2(f-3)t+2f-7. 
\end{eqnarray*}
Denote the last expression of the inequality above by $i_{\langle Q\rangle}(t)$, 
i.e. $i_{\langle Q\rangle }(t):=2(f-1)t^5+2(f-3)t^4+2(f-3)t^3+2(f-3)t^2+2(f-3)t+2f-7$. 
Then 
\begin{eqnarray*}
\frac{\rm d}{{\rm d}t}i_{\langle Q\rangle}(t)&=& 
10(f-1)t^4+8(f-3)t^3+6(f-3)t^2+4(f-3)t+2(f-3) \\
&=&10(f-1)t^4+2(f-3)(4t^3+3t^2+2t+1). 
\end{eqnarray*}
Since $4t^3+3t^2+2t+1>0$ for $-\frac{1}{2}<t<0$ and $f\geq 3$, 
$\frac{\rm d}{{\rm d}t}i_{\langle Q\rangle}(t)$ is positive. 
Thus $i_{\langle Q\rangle}(t)<i_{\langle Q\rangle}(0)=-7<0$ for $-\frac{1}{2}<t<0$. 
Hence $h_{\langle Q\rangle}(t)<0$ for $-\frac{1}{2}<t<0$. \\
\\
\\
\textbf{Proof of Proposition \ref{prop-1} $(2)$}\\ 
By Lemma \ref{lem-2}, we get the lower bound of the derivative of $g_{\langle Q\rangle}(t)$ with respect to $t$ as follows: 
\begin{eqnarray*}
\frac{\rm d}{{\rm d}t}g_{\langle Q\rangle}(t)&=& 
(14c-14)t^6+(12c-12f+12)t^5+(10c+10f-5c_9-40)t^4 \\
&\ &+(16c-16f+4c_9-4c_{10}+16)t^3+(12f-3c_9+3c_{10}-36)t^2 \\
&\ &+(4c-4f+2c_9)t+2f-6 \\
&\geq &14(f-1)t^6+12t^5+20(f-2)t^4+16t^3+12(f-3)t^2+2f-6 \\
&\ &+(-5t^4+4t^3-3t^2+2t)c_9+(-4t^3+3t^2)c_{10}. 
\end{eqnarray*}
Since $(-5t^4+4t^3-3t^2+2t)$ and $(-4t^3+3t^2)$ are positive for $0<t<\frac{1}{2}$ and $f\geq 3$, we obtain 
$\frac{\rm d}{\rm dt}g_{\langle Q\rangle}(t)>0$ for $0<t<\frac{1}{2}$. 

{\flushleft \textbf{Proof of Proposition \ref{prop-1} $(3)$}}\\ 
First we show  
\begin{equation}
4c-4f+c_9-c_{10}+4\geq 0. 
\end{equation}
This inequality is proved as follows. 
By identity $(10)$, we obtain $2c-2f\geq c_9+c_{10}-4$. Thus 
\begin{eqnarray*}
4c-4f+c_9-c_{10}+4&\geq &2(c_9+c_{10}-4)+c_9-c_{10}+4 \\
&=&3c_9+c_{10}-4. 
\end{eqnarray*}
If $3c_9+c_{10}-4$ is negative, then there are five possibilities: 
$(c_9, c_{10})=(0,0)$, $(0,1)$, $(0,2)$, $(0,3)$ or $(1,0)$. But in each of five cases, 
$4c-4f+c_9-c_{10}+4$ is positive because of the inequality (12). Thus we obtain (13). 
We also have to show two inequalities: 
\begin{eqnarray}
2c+2f-c_9-8&>&0, \\
4f-c_9+c_{10}-12&\geq &0. 
\end{eqnarray}
Since $c\geq c_9$ and $f\geq 4$, we obtain the inequality $(14)$. 
By inequality $(11)$, $c\geq c_9$ and $c\geq 4$, 
\begin{eqnarray*}
4f-c_9+c_{10}-12&\geq &2c+8-c_9+c_{10}-12 \\
&\geq &c-c_9+c-4 \\
&\geq &0. 
\end{eqnarray*}
Thus we we obtain the inequality $(15)$. 

Now we are ready to prove Proposition \ref{prop-1} $(3)$. 
The method of proving Proposition \ref{prop-1} $(3)$ is similar to that of proving the first lemma of \cite{ref:ku}. 
By Proposition \ref{prop-1} $(2)$, the real number $r_0\in \left( 0,\frac{1}{2}\right) $ which satisfies 
\begin{eqnarray*}
2&=&2(c-1){r_0}^7+2(c-f+1){r_0}^6+(2c+2f-c_9-8){r_0}^5\\
&\ &+(4c-4f+c_9-c_{10}+4){r_0}^4+(4f-c_9+c_{10}-12){r_0}^3\\
&\ &+(2c-2f+c_9){r_0}^2+(2f-6)r_0 
\end{eqnarray*}
is unique. 
Note that each coefficient of the right-hand side of 
the equality above is non-negative by inequalities (12-15). 
Suppose that there exists a complex number $z$ whose absolute value is less than $r_0$ and satisfying 
$g_{\langle Q\rangle}(z)=0$. 
Denote $z=re^{i\theta }$ where $0<r<r_0$ and $0\leq \theta <2\pi$. Then 
\begin{eqnarray*}
2&=&|g_{\langle Q\rangle}(z)+2| \\
&\leq &2(c-1)|(re^{i\theta })^7|+2(c-f+1)|(re^{i\theta })^6|+(2c+2f-c_9-8)|(re^{i\theta })^5| \\
&\ &+(4c-4f+c_9-c_{10}+4)|(re^{i\theta })^4|+(4f-c_9+c_{10}-12)|(re^{i\theta })^3| \\
&\ &+(2c-2f+c_9)|(r^{i\theta })^2|+(2f-6)|re^{i\theta }| \\
&=&2(c-1)r^7+2(c-f+1)r^6+(2c+2f-c_9-8)r^5+(4c-4f+c_9-c_{10}+4)r^4 \\
&\ &+(4f-c_9+c_{10}-12)r^3+(2c-2f+c_9)r^2+(2f-6)r \\
&=&g_{\langle Q\rangle}(r)+2. 
\end{eqnarray*}
By Proposition \ref{prop-1} $(2)$, 
we obtain $g_{\langle Q\rangle}(r)<g_{\langle Q\rangle}(r_0)=0$. A contradiction. 
Hence $r_0$ is the smallest absolute value among all zeros 
of $g_{\langle Q\rangle}(t)$. 

Now we consider a complex number $z$ whose absolute value is equal to $r_0$. 
Set $z=r_0e^{i\theta }$ and $0\leq \theta <2\pi$. Then $g_{\langle Q\rangle}(z)=0$ implies 
\begin{eqnarray*}
2&=& 2(c-1){r_0}^7\cos 7\theta +2(c-f+1){r_0}^6\cos 6\theta +(2c+2f-c_9-8){r_0}^5\cos 5\theta \\
&\ &+(4c-4f+c_9-c_{10}+4){r_0}^4\cos 4\theta +(4f-c_9+c_{10}-12){r_0}^3\cos 3\theta \\
&\ &+(2c-2f+c_9){r_0}^2\cos 2\theta +(2f-6)r_0\cos \theta \\
&\leq & 2(c-1){r_0}^7+2(c-f+1){r_0}^6+(2c+2f-c_9-8){r_0}^5\\
&\ &+(4c-4f+c_9-c_{10}+4){r_0}^4+(4f-c_9+c_{10}-12){r_0}^3\\
&\ &+(2c-2f+c_9){r_0}^2+(2f-6)r_0 \\
&=&2. 
\end{eqnarray*}
Thus, $\cos k\theta =1$ for $k\in\{1, 5, 6, 7\}$ since the coefficients of $t^k$ of $g_{\langle Q\rangle}(t)$ 
are positive. 
This means that $\theta =0$. Therefore $z=r_0$, and we have proved Proposition \ref{prop-1} $(3)$. \\

In the case that an ideal Coxeter polyhedron is right-angled, 
we can get the explicit value of its growth rate. 
Let $R$ be an ideal right-angled Coxeter polyhedron in $\mathbb{H}^3$. 
Note that $c=c_8$ and $c_9=c_{10}=c_{11}=0$ in this case. 
By substituting $c_8=c$ in identity (9), we obtain $c=f-2$. 
Thus 
\begin{eqnarray*}
g_{\langle R\rangle}(t)&=&(2f-6)t^7-2t^6+(4f-12)t^5-4t^4+(4f-12)t^3-4t^2+(2f-6)t-2 \\
&=&2(t^2+1)(t^4+t^2+1)((f-3)t-1). 
\end{eqnarray*}
Therefore $g_{\langle R\rangle}(t)=0$ has the unique solution $t=\frac{1}{f-3}$ 
which has the smallest absolute value. 
Since its growth rate $(f-3)$ is a natural number, it is also a Perron number. 
Note that the solution $t=\frac{1}{f-3}$ of $g_{\langle R\rangle}(t)=0$ 
is in $(0,\frac{1}{2})$ because an ideal right-angled polyhedron in $\mathbb{H}^3$ has 
at least eight faces (see the last paragraph of Section 2). 

Combining what we have seen above, Proposition \ref{prop-1} and two inequalities 
$g_{\langle P\rangle}(0)<0$, $g_{\langle P\rangle}(\frac{1}{2})>0$, we obtain the following proposition. 
\begin{prop}\label{prop-2} 
Let $P$ be an ideal Coxeter polyhedron in $\mathbb{H}^3$. 
The equation $g_{\langle P\rangle}(t)=0$ has the unique solution $t=r_0\in \left( 0,\frac{1}{2}\right)$ 
whose absolute value is smaller than any other solution. 
In addition, the slope $\left. \frac{\rm d}{\rm dt}g_{\langle P\rangle}(t)\right|_{t=r_0}$ is positive. 
\end{prop}

Denote its growth rate $\frac{1}{t}$ by $\tau $ for the rest of this section. 
Now we want to prove that $\tau $ is an algebraic integer to complete the proof of Theorem \ref{thm-3}. 
Since $g_{\langle P\rangle }\left( \frac{1}{\tau }\right) =0$, 
$-\frac{\tau ^7}{2}g_{\langle P\rangle }\left( \frac{1}{\tau }\right) =0$. That is to say, 
\begin{eqnarray}
& &\tau ^7-(f-3)\tau ^6-\left( c-f+\frac{c_9}{2}\right) \tau ^5-\left( 2f-\frac{c_9}{2}+\frac{c_{10}}{2}-6\right) \tau ^4 
\nonumber \\
& &-\left( 2c-2f+\frac{c_9}{2}-\frac{c_{10}}{2}+2\right) \tau ^3 \nonumber \\
& &-\left( c+f-\frac{c_9}{2}-4\right) \tau ^2-(c-f+1)\tau ^2-c+1=0. 
\end{eqnarray}

To prove that $\tau $ is an algebraic integer, 
we need to show that any coefficient of $t^i$ of $g_{\langle P\rangle }(t)$ $(i=1, \cdots , 7)$ 
is even. 
This is satisfied if and only if both $c_9$ and $c_{10}$ are even numbers. 
By identity $(6)$, $c_{11}$ is even. Then, by $(4)$, $c_9$ is also even. 
Therefore, by $(10)$, we show that both $c_9$ and $c_{10}$ are even numbers. 

Note that any coefficient of $\tau ^i$ $(i=1, \cdots , 7)$ of the equation above is integer. 
Thus, the growth rate $\tau $ is an algebraic integer. 
By Lemma \ref{thm:And}, $t=\frac{1}{\tau }$ is the unique solution of $g_{\langle P\rangle }(t)=0$ 
which has the smallest absolute value. Therefore $\tau $ is the unique solution 
which has the biggest absolute value of the solution of the equation $(14)$. 
Hence we proved Theorem \ref{thm-3}. 

\section{The minimal growth rate of ideal Coxeter polyhedra}
In this section, we find an ideal Coxeter polyhedron which has the minimal growth rate. 
First of all, we compare the growth rates of ideal Coxeter simplices. 
As we have already explained in Section 3, 
there exist only three ideal Coxeter simplices up to isometry (see also Fig.1). 
One is a simplex which has two $\frac{\pi}{2}$-edges, two $\frac{\pi}{3}$-edges 
and two $\frac{\pi}{6}$-edges. We denote this polyhedron by $P_1$. 
Denote by $P_2$ another ideal Coxeter simplex
which has exactly two $\frac{\pi }{2}$-edges and four $\frac{\pi }{4}$-edges. 
We denote the last one by $P_3$. 
Since $P_1$ satisfies $f=c=c_{11}=4$ and $c_9=c_{10}=0$, we obtain 
\begin{eqnarray*}
g_{\langle P_1\rangle}(t)&=&6t^7+2t^6+8t^5+4t^4+4t^3+2t-2 \\
&=&2(1+t^2)(-1+t+t^2+t^3+t^4+3t^5). 
\end{eqnarray*}
Then its growth function ${\cal F}_{\langle P_1\rangle}(t)$ is as follows: 
\begin{equation*}
{\cal F}_{\langle P_1\rangle}(t)=\frac{(1+t)^2(1+t+t^2)(1-t+t^2)}{(1-t)(1-t-t^2-t^3-t^4-3t^5)}.
\end{equation*} 
Thus, The growth rate of $P_1$ is $\tau(P_1)\approx 2.030735$. 

The simplex $P_2$ satisfies $f=c=c_{10}=4$ and $c_9=0$. 
Therefore we obtain 
\begin{eqnarray*}
g_{\langle P_2\rangle}(t)&=&6t^7+2t^6+8t^5+8t^3+2t-2 \\
&=&2(1+t+t^2)(1-t+t^2)(-1+t+t^2+3t^2). 
\end{eqnarray*}
Thus, we obtain its growth function ${\cal F}_{\langle P_2\rangle}(t)$ as follows: 
\begin{equation*}
{\cal F}_{\langle P_2\rangle}(t)=\frac{(1+t)^2(1+t^2)}{(1-t)(1-t-t^2-3t^3)}.
\end{equation*} 
Thus, its growth rate $\tau(P_2)$ is roughly $2.130395$. 

In the same manner as above, we can calculate the growth rate of $P_3$ denoted by $\tau (P_3)$. 
Because $P_3$ satisfies $f=c=c_9=4$ and $c_{10}=0$, 
we obtain 
\begin{eqnarray*}
g_{\langle P_3\rangle}(t)&=&6t^7+2t^6+4t^5+8t^4+4t^2+2t-2 \\
&=&2(1+t)(1+t^2)(1-t+t^2)(-1+t+3t^2). 
\end{eqnarray*}
Thus, the growth function ${\cal F}_{\langle P_3\rangle}(t)$ is represented as follows: 
\begin{equation*}
{\cal F}_{\langle P_3\rangle}(t)=\frac{(1+t)(1+t+t^2)}{(1-t)(1-t-3t^2)}.
\end{equation*}
And then we obtain the growth rate $\tau (P_3)\approx 2.302776$. 
Hence $\tau (P_1)<\tau (P_2)<\tau (P_3)$. 

Next, we show that the growth rate of any ideal Coxeter polyhedron which is not a simplex is bigger than 
that of $P_3$. 
Let $P'$ be an ideal Coxeter polyhedron which is not a simplex. 
From now on, we denote the number of faces of $P'$ by $f(P')$ and the number of cusps of $P'$ by $c(P')$. 
We also denote the number of cusps of type $(ii)$ $($resp. $(iii))$ by $c_9(P')$ $($resp. $c_{10}(P'))$. 

By Proposition \ref{prop-2}, we have only to consider 
the solution of $g_{\langle P\rangle}(t)=0$ in the interval $\left( 0,\frac{1}{2}\right)$ 
in order to determine the growth rate of an ideal Coxeter polyhedron. 

Since $P'$ is not a simplex, $f(P')\geq 5$ and $c(P')\geq 5$. 
Thus, 
\begin{eqnarray*}
g_{\langle P'\rangle}(t)-g_{\langle P_3\rangle}(t)&=&2t^2(1+t^2+t^3+t^4+t^5)(c(P')-4)\\
&\ &+2t(1-t)(1+t^2)^2(f(P')-4) +t^2(1-t)(1+t^2)(c_9(P')-4)\\
&\ &+t^3(1-t)c_{10}(P') \\
&\geq &2t^2(1+t^2+t^3+t^4+t^5)+2t(1-t)(1+t^2)^2-4t^2(1-t)(1+t^2) \\
&=&2t(1-2t+4t^2-3t^3+4t^4+t^6) 
\end{eqnarray*}
for $0<t<\frac{1}{2}$. 
Since $-3t^3+4t^2-2t+1$ is monotonically decreasing, $-3t^3+4t^2-2t+1>\frac{5}{8}$ for $0<t<\frac{1}{2}$. 
Thus $g_{\langle P'\rangle}(t)-g_{\langle P_3\rangle}(t)>0$ for $0<t<\frac{1}{2}$. 
That is to say, the growth rate of any ideal Coxeter polyhedron which is not a simplex is greater than 
that of any ideal Coxeter simplex. 
Hence we obtain the following theorem. 
\begin{thm}\label{thm-4} 
An ideal Coxeter polyhedron in $\mathbb{H}^3$, which has the minimal growth rate is $P_1$.  
\end{thm}
\begin{rem}
{\rm Note that the second (resp. the third) smallest growth rate of hyperbolic ideal Coxeter polyhedron is 
$\tau (P_2)$ (resp. $\tau (P_3)$). }
\end{rem}

\section{Growth rates and volumes of some ideal Coxeter polyhedra}
Some of known examples suggest that if the growth rates of an ideal Coxeter polyhedron is large, 
then its volume is also large. 
In this section, we see the relation between 
the growth rates and volumes of some ideal Coxeter polyhedra. 
To calculate the volume of some ideal Coxeter polyhedra, we use the following theorem. 
\begin{thm}[\cite{ref:mil}]\label{thm-5}
Let $T$ be an ideal tetrahedron satisfying that the dihedral angle at edge is equal to that of the opposite 
edge. 
Denote the different dihedral angles of $T$ by $\alpha $, $\beta $ and $\gamma $. 
We also denote the volume of $T$ by {\rm vol}$(T)$. Then 
\begin{equation*}
{\rm vol}(T)=\Lambda (\alpha )+\Lambda (\beta )+\Lambda (\gamma ), 
\end{equation*}
where $\Lambda$ is a Lobachevskij function defined by 
\begin{equation*}
\Lambda (x)=-\int_{0}^x\log |2\sin \zeta |d\zeta .
\end{equation*}
\end{thm}
The first derivative and the second derivative of $\Lambda (x)$ with respect to $x$ are as follows: 
\begin{eqnarray*}
\frac{\rm d}{{\rm d}x}\Lambda (x)&=&-\log (2\sin x), \\
\frac{\rm d^2}{{\rm d}x^2}\Lambda (x)&=&-\frac{1}{\tan x}. 
\end{eqnarray*}
Note that these derivatives are negative 
for $x\in \left( \frac{\pi}{6},\frac{\pi}{2}\right) $ and that $\Lambda \left(\frac{\pi }{2}\right) =0$. 
Thus, the graph of $y=\Lambda (x)$ for $\frac{\pi}{6}\leq x\leq \frac{\pi}{2}$ is depicted as in Fig.2.  
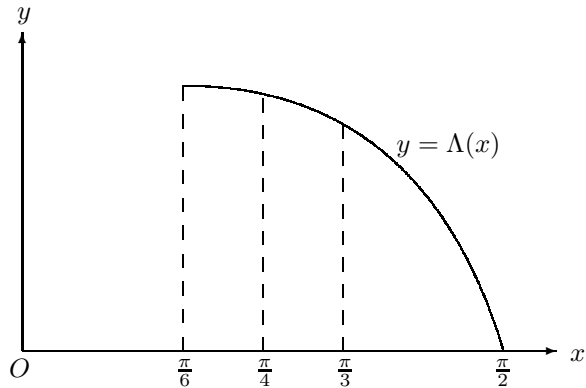
\begin{figure}[htp]
\begin{center}
\begin{picture}(220, 130)
\put(0,0){\vector(1,0){200}}
\put(0,0){\vector(0,1){120}}
\qbezier(60,100)(150,100)(180,0)
\multiput(60,0)(0,10.5){10}{\line(0,1){5}}
\multiput(90,0)(0,10){10}{\line(0,1){5}}
\multiput(120,0)(0,10){9}{\line(0,1){5}}
\put(57,-10){$\frac{\pi}{6}$}
\put(87,-10){$\frac{\pi}{4}$}
\put(117,-10){$\frac{\pi}{3}$}
\put(176,-10){$\frac{\pi}{2}$}

\put(140,75){$y=\Lambda (x)$} 

\put(-5,-10){$O$}
\put(205,-4){$x$}
\put(-2,125){$y$}
\end{picture}
\end{center}
\caption{The graph of $y=\Lambda (x)$ for $\frac{\pi}{6} \leq x\leq \frac{\pi}{2}$}
\end{figure} 
\clearpage
\begin{flushleft}By theorem \ref{thm-5}, we obtain \end{flushleft}
\begin{eqnarray*}
{\rm vol}(P_1)&=& \Lambda \left( \frac{\pi}{3}\right) +\Lambda \left( \frac{\pi}{6}\right) ,\\
{\rm vol}(P_2)&=& 2\Lambda \left( \frac{\pi}{4}\right) ,\\
{\rm vol}(P_3)&=& 3\Lambda \left( \frac{\pi}{3}\right) .
\end{eqnarray*}
Since $\frac{\rm d}{{\rm d}x}\Lambda (x)$ and $\frac{\rm d^2}{{\rm d}x^2}\Lambda (x)$ 
are negative for $x\in \left( \frac{\pi}{6},\frac{\pi}{2}\right)$, 
$\Lambda $ is a strictly concave function. 
Thus we get 
\begin{eqnarray*}
\frac{2\Lambda \left( \frac{\pi}{4}\right) +\Lambda \left( \frac{\pi}{2}\right) }{3}
&<&\Lambda \left( \frac{\pi}{3}\right) ,\\
\frac{\Lambda \left( \frac{\pi}{3}\right) +\Lambda \left( \frac{\pi}{6}\right) }{2}
&<&\Lambda \left( \frac{\pi}{2}\right) .
\end{eqnarray*}
Thus ${\rm vol}(P_1)<{\rm vol}(P_2)<{\rm vol}(P_3)$. 
Note that $\tau (P_1)<\tau (P_2)<\tau (P_3)$. 

Now we consider the other ideal Coxeter polyhedra. 
A polyhedron that has a base and at least three triangular faces 
that meet at a point above the base is called a pyramid. 
Since any cusp of a three-dimensional ideal Coxeter polyhedron  
is shared by three or four faces, a three-dimensional ideal Coxeter pyramid must be either a tetrahedron 
or a square pyramid. By Lemma \ref{lem-1}, there exist only two ideal Coxeter square pyramids. 
One is a Coxeter pyramid whose square face has two opposite $\frac{\pi}{3}$-edges and two opposite 
$\frac{\pi}{6}$-edges. 
The other is a Coxeter pyramid whose square face has four $\frac{\pi}{4}$-edges. 
Their combinatorial structures are depicted in Fig. 3. 
 
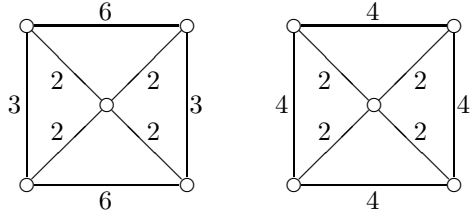
\begin{figure}[htp]
\begin{center}
\begin{picture}(160, 80)
\put(0,0){\circle{5}} 
\put(0,60){\circle{5}}
\put(60,0){\circle{5}}
\put(60,60){\circle{5}}
\put(30,30){\circle{5}}
\put(3,0){\line(1,0){54}}
\put(0,3){\line(0,1){54}}
\put(3,60){\line(1,0){54}}
\put(60,3){\line(0,1){54}}

\put(2,2){\line(1,1){26}}
\put(2,58){\line(1,-1){26}}
\put(58,58){\line(-1,-1){26}}
\put(58,2){\line(-1,1){26}}

\put(-7,27){$3$}
\put(27,-9){$6$}
\put(27,62){$6$}
\put(61,27){$3$}
\put(9,36){$2$}
\put(9,17){$2$}
\put(45,36){$2$}
\put(45,17){$2$}

\put(100,0){\circle{5}} 
\put(100,60){\circle{5}}
\put(160,0){\circle{5}}
\put(160,60){\circle{5}}
\put(130,30){\circle{5}}
\put(103,0){\line(1,0){54}}
\put(100,3){\line(0,1){54}}
\put(103,60){\line(1,0){54}}
\put(160,3){\line(0,1){54}}

\put(102,2){\line(1,1){26}}
\put(102,58){\line(1,-1){26}}
\put(158,58){\line(-1,-1){26}}
\put(158,2){\line(-1,1){26}}

\put(93,27){$4$}
\put(127,-9){$4$}
\put(127,62){$4$}
\put(161,27){$4$}
\put(109,36){$2$}
\put(109,17){$2$}
\put(145,36){$2$}
\put(145,17){$2$}

\end{picture}
\end{center}
\caption{Combinatorial structures of ideal Coxeter pyramids in $\mathbb{H}^3$}
\end{figure}

Denote by $P_4$ the first ideal Coxeter pyramid, and denote by $P_5$ the last one. 
Consider a hyperplane whose closure includes the cusp of $P_4$ shared 
by its four faces and the other two cusps which are opposite in the square face of $P_4$. 
Decompose $P_4$ by this hyperplane, and we obtain two simplices isometric to $P_1$. Thus 
\begin{equation*}
{\rm vol}(P_4)=2{\rm vol}(P_1)=2\Lambda \left( \frac{\pi}{3}\right) +2\Lambda \left( \frac{\pi}{6}\right) .
\end{equation*}
By decomposing $P_5$ as in the case $P_4$, we obtain two simplices isometric to $P_2$. Thus 
\begin{equation*}
{\rm vol}(P_5)=2{\rm vol}(P_2)=4\Lambda \left( \frac{\pi }{4}\right) . 
\end{equation*} 
Thus ${\rm vol}(P_4)<{\rm vol}(P_5)$. 
Since $\frac{\rm d}{{\rm d}x}\Lambda (x)<0$ for $x\in \left( \frac{\pi}{6},\frac{\pi}{2}\right)$ and 
$\Lambda \left( \frac{\pi}{2}\right) =0$, 
we obtain $\Lambda \left(\frac{\pi}{6}\right) >\Lambda \left( \frac{\pi}{3}\right)$. 
Thus ${\rm vol}(P_4)>{\rm vol}(P_3)$. 
Hence 
\begin{equation*}
{\rm vol}(P_1)<{\rm vol}(P_2)<{\rm vol}(P_3)<{\rm vol}(P_4)<{\rm vol}(P_5). 
\end{equation*}
On the other hand, the growth rate of $P_4$ is $\tau (P_4)\approx 2.747380$ and that of $P_5$ is 
$\tau (P_5)\approx 2.845466$. 
Thus 
\begin{equation*}
\tau (P_1)<\tau (P_2)<\tau (P_3)<\tau (P_4)<\tau (P_5). 
\end{equation*}
Moreover, the following is true. 
\begin{thm}\label{thm-6} 
Let $P$ and $P'$ be ideal Coxeter polyhedra in $\mathbb{H}^3$ 
and $\tau (P)$ $($resp. $\tau (P'))$ the growth rate of $P$ $($resp. $P')$. 
Suppose that $P$ has a face $F$ which is isometric to a face $F'$ of $P'$. 
By gluing $F$ and $F'$, we obtain the polyhedron. 
Denote this polyhedron by $P\ast _FP'$. 
If $P\ast _FP'$ is also an ideal Coxeter polyhedron in $\mathbb{H}^3$, 
then its growth rate $\tau (P\ast _FP')$ is greater than both $\tau (P)$ and $\tau (P')$. 
\end{thm}
\begin{proof}
Denote by $c_{m,n}(F)$ the number of cusps which are shared 
by both $\frac{\pi}{m}$-edge and $\frac{\pi}{n}$-edge $F$ has. 
If $F$ has a $\frac{\pi }{4}$-egde, this edge must coincide with an $\frac{\pi }{4}$-edge of $F'$ since 
$P\ast_ FP'$ is a Coxeter polyhedron. 
Note that the dihedral angle of $P\ast _FP'$ at this edge is $\frac{\pi }{2}$. 
Thus, that any cusp counted in $c_{4,4}(F)$ coincides to a cusp of $P'$ which is counted in $c_{4,4}(F')$. 
In addition, the type of this cusp in $P\ast _FP'$ is $(i)$, i.e. this cusp is counted in $c_8(P\ast _FP')$. 

Let $F_1$ be a $P$'s face which is adjacent to $F$ at $\frac{\pi }{2}$-edge, and 
$F_1'$ be a $P'$'s face which is adjacent to $F'$ at $\frac{\pi }{2}$-edge. 
If these two $\frac{\pi }{2}$-edges are identified by gluing, 
then these $\frac{\pi }{2}$-edges vanish in $P\ast _FP'$. 
In this case, the hyperplane containing $F_1$ is the same as that containing $F_1'$. 
From the above, we can realize two things. 
One is that any cusp counted in $c_{2,2}(F)$ coincides with both a cusp counted in $c_{2,2}(F')$ 
and a cusp counted in $c_8(P\ast _FP')$ since $P\ast _FP'$ is also Coxeter polyhedron. 
The other is that any cusp counted in $c_{2,4}(F)$ must coincide with a cusp counted in both $c_{2,4}(F')$ and  
$c_{10}(P\ast _FP')$. 

The dihedral angle of $P\ast _FP'$ arisen from a $\frac{\pi }{6}$-edge of $F$ and 
a $\frac{\pi }{6}$-edge of $F'$ is $\frac{\pi }{3}$. 
Thus, by gluing $F$ and $F'$, 
any cusp counted in $c_{2,6}(F)$ coincides with a cusp counted in $c_{2,6}(F')$, 
and it is also counted in $c_9(P\ast _FP')$. 
Since $\frac{\pi}{6}+\frac{\pi}{3}=\frac{\pi}{2}$, 
any cusp counted in $c_{3,6}(F)$ coincides with a cusp counted in $c_{6,3}(F')$ and in $c_8(P\ast _FP')$. 

By the above observation on the number of each kind of cusps, we obtain 
\begin{eqnarray}
c(P\ast _FP')&=&c(P)+c(P')-c(F), \\
c_9(P\ast _FP')&=&c_9(P)+c_9(P')+c_{2,6}(F), \\
c_{10}(P\ast _FP')&=&c_{10}(P)+c_{10}(P')-2c_{4,4}(F)-c_{2,4}(F). 
\end{eqnarray}
We denote by $e_k(F)$ the number of $\frac{\pi}{k}$-edges $F$ has. 
Since any edge of $\frac{\pi }{2}$-edge which $F$ has vanishes in $P\ast _FP'$, we get 
\begin{equation}
f(P\ast _FP')=f(P)+f(P')-e_2(F)-2. 
\end{equation}
Then, by $(15$-$18)$, 
\begin{eqnarray*}
& &g_{\langle P\ast _FP'\rangle }(t)-g_{\langle P\rangle }(t) \\
&=&(2t^2+4t^3+4t^4+2t^5+2t^6+2t^7)(c(P\ast _FP')-c(P)) \\
& &+(2t-2t^2+4t^3-4t^4+2t^5-2t^6)(f(P\ast _FP')-f(P)) \\
& &+(t^2-t^3+t^4-t^5)(c_9(P\ast _FP')-c_9(P)) \\
& &+(t^3-t^4)(c_{10}(P\ast _FP')-c_{10}(P)) \\
&=&(2t^2+4t^3+4t^4+2t^5+2t^6+2t^7)(c(P')-c(F)) \\
& &+(2t-2t^2+4t^3-4t^4+2t^5-2t^6)(f(P')-e_2(F)-e_4(F)+e_4(F)-2) \\
& &+(t^2-t^3+t^4-t^5)(c_9(P')+c_{2,6}(F)) \\
& &+(t^3-t^4)(c_{10}(P')-2c_{4,4}(F)-c_{2,4}(F)) \\
&=&(2t^2+4t^3+4t^4+2t^5+2t^6+2t^7)(c(P')-c(F)) \\
& &+(2t-2t^2+4t^3-4t^4+2t^5-2t^6)(f(P')-e_2(F)-e_4(F)-2) \\
& &+(t^2-t^3+t^4-t^5)(c_9(P')+c_{2,6}(F)) \\
& &+(t^3-t^4)(c_{10}(P')-c_{4,4}(F)-c_{2,4}(F)) \\
& &+(2t-2t^2+4t^3-4t^4+2t^5-2t^6)e_4(F)-(t^3-t^4)c_{4,4}(F). 
\end{eqnarray*}
It is clear that, for $0<t<\frac{1}{2}$, 
$(2t^2+4t^3+4t^4+2t^5+2t^6+2t^7)(c(P')-c(F))$ is positive, and 
$(2t-2t^2-4t^4+2t^5-2t^6)$, $(t^2-t^3+t^4-t^5)(c_9(P')+c_{2,6}(F))$ and 
$(t^3-t^4)(c_{10}(P')-c_{4,4}(F)-c_{2,4}(F))$ are greater than or equal to $0$. 
By the definitions of $c_{4,4}(F)$ and $c_{2,4}(F)$, we obtain an identity; 
\begin{equation*}
2e_4(F)=2c_{4,4}(F)+c_{2,4}(F). 
\end{equation*}
Since $2t-2t^2-4t^4+2t^5-2t^6$ and $t^3-t^4$ are also positive for $0<t<\frac{1}{2}$, 
by using the above identity, 
we get  
\begin{eqnarray*}
& &(2t-2t^2+4t^3-4t^4+2t^5-2t^6)e_4(F)-(t^3-t^4)c_{4,4}(F) \\
&=&(2t-2t^2+4t^3-4t^4+2t^5-2t^6)(c_{4,4}(F)+\frac{1}{2}c_{2,4}(F))-(t^3-t^4)c_{4,4}(F) \\
&\geq &t(2-2t-t^2+3t^3-3t^3+2t^4-2t^5)c_{4,4}(F) \\
&\geq &0
\end{eqnarray*}
for $0<t<\frac{1}{2}$. 
From the above, if $f(P')-e_2(F)-e_4(F)-2\geq 0$, then 
$g_{\langle P\ast _FP'\rangle }(t)>g_{\langle P\rangle }(t)$ for $0<t<\frac{1}{2}$. 

From now on, we consider the case $f(P')-e_2(F)-e_4(F)-2<0$. 
Since the number of faces which are adjacent to $F'$ is at least $e_2(F)+e_4(F)$, 
we obtain $f(P')\geq e_2(F)+e_4(F)+1$. 
Thus, we have only to consider the case $f(P')=e_2(F)+e_4(F)+1$. 
That is to say, $P'$ has only faces which are adjacent to $F'$ at 
either a $\frac{\pi}{2}$-edge or a $\frac{\pi}{4}$-edge other than $F'$. 
If $F'$ has a cusp shared by four $\frac{\pi}{2}$-edge, then there is a face which is parallel to $F'$. 
Thus, $F'$ does not have such a cusp. 
Then any cusp of $F$ is shared exactly three edges: two $\frac{\pi}{4}$-edges and one $\frac{\pi}{2}$-edge. 
Therefore, $e_4(F)\geq 2$ and $c_{10}(P')\geq 3$. 
By these two inequalities, we obtain the following lower bound for $0<t<\frac{1}{2}$.  
\begin{eqnarray*}
g_{\langle P\ast _FP'\rangle }(t)-g_{\langle P\rangle }(t)
&=&(2t^2+4t^3+4t^4+2t^5+2t^6+2t^7)(c(P')-c(F)) \\
& &+(2t-2t^2+4t^3-4t^4+2t^5-2t^6)(f(P')-e_2(F)-2) \\
& &+(t^2-t^3+t^4-t^5)(c_9(P')+c_{2,6}(F)) \\
& &+(t^3-t^4)(c_{10}(P')-2c_{4,4}(F)-c_{2,4}(F)) \\
&>&(2t-2t^2+4t^3-4t^4+2t^5-2t^6)(e_4(F)-1) \\
& &+(t^3-t^4)(c_{10}(P')-2e_4(F)) \\
&\geq &(2t-2t^2+4t^3-4t^4+2t^5-2t^6)(e_4(F)-1) \\
& &+(t^3-t^4)(3-2e_4(F)) \\
&\geq &(2t-2t^2+2t^3-2t^4+2t^5-2t^6)e_4(F) \\
& &-2t+2t^2-t^3+t^4-2t^5+2t^6 \\
&\geq &2(2t-2t^2+2t^3- 2t^4+2t^5-2t^6) \\
& &-2t+2t^2-t^3+t^4-2t^5+2t^6 \\
&=&2t-2t^2+3t^3-3t^4+2t^5-2t^6 \\
&>&0. 
\end{eqnarray*}

Thus, in each case, $g_{\langle P\ast _FP'\rangle }(t)>g_{\langle P\rangle }(t)$ for $0<t<\frac{1}{2}$. 
Hence, by Proposition \ref{prop-2}, we obtain 
\begin{equation*}
\tau (P\ast _FP')>\max \{ \tau (P),\ \tau (P')\} .
\end{equation*}
\end{proof}

\subsection*{Acknowledgment}
The author would like to show his greatest appreciation to Professor Hiroyasu Izeki 
who provided valuable comments and suggestions. 



\vspace{1cm}
\begin{flushleft}
Waseda University Senior High School, \\
3-31-1 Kamishakujii, Nerima-ku, Tokyo, 177-0044, Japan. \\
email: \verb|jun-b-nonaka@waseda.jp|
\end{flushleft}
\end{document}